\documentclass[11pt,a4paper,oneside]{amsart}
\pdfoutput=1
\usepackage{amsmath,amsfonts,amsthm,amssymb,latexsym,graphicx,color,vmargin}
\usepackage{hyperref}

\newtheorem{theorem}{Theorem}

\newtheorem{lemma}{Lemma}

\theoremstyle{definition}
\newtheorem{definition}{Definition}

\newcommand{\beql}[1]{\begin{equation}\label{#1}}
\newcommand{\eeq}{\end{equation}}
\newcommand{\comment}[1]{}

\newcommand{\Abs}[1]{{\left|{#1}\right|}}
\newcommand{\Lone}[1]{{\left\|{#1}\right\|_{L^1}}}

\newcommand{\Mean}{{\bf E}}

\newcommand{\Prob}[1]{{{\bf{Pr}}\left[{#1}\right]}}
\newcommand{\Set}[1]{{\left\{{#1}\right\}}}

\newcommand{\RR}{{\mathbb R}}
\newcommand{\CC}{{\mathbb C}}
\newcommand{\ZZ}{{\mathbb Z}}

\newcommand{\One}[1]{{\mathbf 1}\left(#1\right)}

\newcommand{\ft}[1]{\widehat{#1}}

\newcounter{rem}
\newcounter{othm}
\def\theothm{\Alph{othm}} 
\newenvironment{othm}{
  \sf
  \vskip 0.10in
  \refstepcounter{othm}
  \noindent{\bf Theorem\ \theothm}
}{\vskip 0.10in}

\newcounter{open}
\begin{document}

\title[The discrepancy of a needle on a checkerboard, II]{The discrepancy of a needle on a checkerboard, II}

\author[A. Iosevich]{{Alex Iosevich}}
\address{A.I.: Department of Mathematics, University of Missouri, Columbia MO 65211-4100, U.S.A}
\email{iosevich@math.missouri.edu}

\author[M. Kolountzakis]{{Mihail N. Kolountzakis}}
\address{M.K.: Department of Mathematics, University of Crete, Knossos Ave., GR-714 09, Iraklio, Greece}
\email{kolount@gmail.com}

\thanks{
MK: Supported by research grant No 2569 from the Univ.\ of Crete.
}

\date{November 2008}

\begin{abstract}
Consider the plane as a checkerboard, with each unit square
colored black or white in an arbitrary manner.
In a previous paper we showed that for any such coloring there are straight line segments, of arbitrarily large length, such that the difference of their white length minus their black length, in absolute value, is at least the square root of their length, up to a multiplicative constant.
For the corresponding ``finite'' problem ($N \times N$ checkerboard) we had proved that we can
color it in such a way that the above quantity is at most $C \sqrt{N \log N}$,
for any placement of the line segment. In this followup we show that it is possible to color the infinite checkerboard with two colors so that for any line segment $I$ the excess of one color over another is bounded above by $C_\epsilon \Abs{I}^{\frac12+\epsilon}$, for any $\epsilon>0$. We also prove lower bounds for the discrepancy of circular arcs. Finally, we make some observations regarding the $L^p$ discrepancies for segments and arcs, $p<2$, for which our $L^2$-based methods fail to give any reasonable estimates.
\end{abstract}

\maketitle

\tableofcontents

\section{Introduction to checkerboard discrepancy}

In a previous paper \cite{kolountzakis} we answered a question posed to us by
\href{http://users.uoa.gr/~ppapazog/}{P. Papasoglu} \cite{papasoglu}:
\begin{figure}[h]
 \begin{center} \resizebox{5cm}{!}{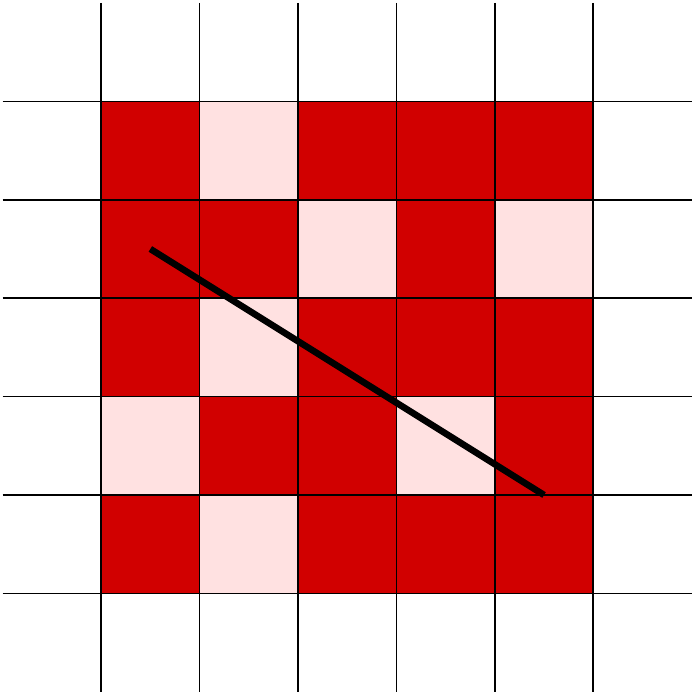} \end{center}
\caption{A colored checkerboard with a needle on it.}
\label{fig:board}
\end{figure}
\begin{quote}
Consider the plane as a checkerboard, with each unit square
$$
[m, m+1) \times [n,n+1),\ \  m,n \in \ZZ,
$$
colored black or white.
Is it possible that there is such a coloring and a finite constant $M$ such that
for any line segment $I$ placed on the checkerboard the difference of its white length minus
its black length is, in absolute value, at most $M$? (See Fig.\ \ref{fig:board}.)
\end{quote}

We proved in \cite{kolountzakis} that the answer to the above question is negative. In particular we
showed that for any checkerboard coloring there are aribtrarily large straight line segments $I$
whose discrepancy with respect to the given checkerboard coloring
(the excess of black vs white length, in absolute value)
is at least $C \sqrt{\Abs{I}}$, where $C$ is an absolute positive constant.
Our approach was Fourier analytic and used rather strongly the fact that the object in question was
straight.

In this paper we extend our results on checkerboard discrepancy to circular arcs:
\begin{theorem}\label{th:arcs}
Suppose that the function $f:\RR^2 \to \Set{-1, +1}$ is constant in each cell $x+[0,1)^2$, $x \in \ZZ^2$.
Then for arbitrarily large $R>0$ there is a circular arc $I$ of radius $R$ such that
\beql{discrepancy-lower-bound}
\Abs{ \int_I f } \ge C \sqrt{R},
\eeq
where $C>0$ is an absolute constant.
\end{theorem}
We give the proof of Theorem \ref{th:arcs} in \S\ref{sec:lower}.

In \cite{kolountzakis} we also showed how to construct a $N \times N$ checkerboard coloring
so that the discrepancy of any straight line segment with respect to that coloring is $O(\sqrt{N \log N})$.
The method was simply a randomized assignment of the colors, and we remarked in \cite{kolountzakis}
that we could not deduce from that the existence of an infinite checkerboard coloring (the entire plane)
with respect to which the discrepancy of any straight line segment $I$ was  $o({\Abs{I}})$.
We rectify this situation in \S\ref{sec:upper} where we show how to construct such a good coloring of the infinite plane.
\begin{theorem}\label{th:upper}
There is a function $f:\RR^2 \to \Set{-1, +1}$, constant in each cell $x+[0,1)^2$,
$x \in \ZZ^2$, such that, for each $\epsilon>0$ and line segment $I$
\beql{discr-ub}
\Abs{\int_I f} \le C_\epsilon \Abs{I}^{\frac12+\epsilon}.
\eeq
Here $C_\epsilon$ is a constant that depends on $\epsilon$ only.
\end{theorem}

The questions dealt with in this paper fall naturally into the subject of
geometric discrepancy \cite{beck-chen,matousek}.
In this research area there is usually an underlying measure $\mu$ as well as a family ${\mathcal F}$ of
allowed subsets of Euclidean space, on which the measure $\mu$ is evaluated and upper and lower bounds
are sought on the range of $\mu$ on ${\mathcal F}$.
The most classical case is that where $\mu$ is a normalized collection of points masses in the unit square minus Lebesgue
measure and ${\mathcal F}$ consists of all axis-aligned rectangles in the unit square.
Usually the underlying measure $\mu$ has an atomic part (point masses) and the family ${\mathcal F}$
consists of ``fat'' sets.
In the problem we are studying here the measure $\mu$ has no atomic part (it is absolutely continuous)
and the collection ${\mathcal F}$ consists of all straight line segments or circular arcs, which may be considered thin sets,
and, strictly speaking, $\mu$ is $0$ on these sets.
The work in the bibliography which appears to be most related to this paper is that of Rogers \cite{rogers} where the measure $\mu$ is the same as here but the family ${\mathcal F}$ consists not of straight line segments but of thin strips.

\section{Proof of the lower bound for circular arcs}
\label{sec:lower}

Here we prove Theorem \ref{th:arcs}. It follows directly from the following theorem.
\begin{theorem}\label{th:circles}
Suppose the function $f:\RR^2 \to \Set{-1, +1}$ is zero outside
the square $[0,N]\times[0,N]$ and is constant on each of
the squares $(i,j)+[0,1)^2$, $0\le i, j<N$.
Then there is a circle $K$ of radius $R$, $N/5 < R < N/4$, such that
$$
\Abs{\int_K f} \ge C N^{1/2},
$$
where $C$ is a positive constant.
\end{theorem}

Let us remark here that the fact that we obtain, in Theorem \ref{th:arcs}, arcs and not full circles of large discrepancy, is a byproduct of looking at a finite part of the infinite checkerboard. In other words, the circle of large discrepancy in the $N \times N$ checkerboard which is proved to exist in Theorem \ref{th:circles}, translates only to an arc in the infinite checkerboard as it may not entirely line inside the $N \times N$ square.

\begin{proof}
Let $\sigma_t$ be arc-length measure on a circle of center $0$ and radius $t$. We have the estimate (see e.g.\ \cite{wolff})
\beql{asymptotics}
\ft{\sigma_1}(x) = 2 r^{-1/2} \cos\left(2\pi r - \frac{\pi}{4}\right) + O(r^{-3/2}),
	\ \ (\mbox{ as $r=\Abs{x}\to\infty$}).
\eeq
We also have $\ft{\sigma_t}(\xi) = t \ft{\sigma_1}(t\xi)$.
\begin{lemma}
For any constants $c_0 > 0, c_1 > 1$ there is a constant $c_2>0$ such that for all $x>c_0$ we have
\beql{constant-lb}
\int_x^{c_1 x} \Abs{\ft{\sigma_1}(u)}^2\,du \ge c_2.
\eeq
\end{lemma}
\begin{proof}
From \eqref{asymptotics} there is a constant $A>0$ such that for $x>A$ we have (again $r = \Abs{x}$)
$$
\Abs{\ft{\sigma_1}(x)}^2 \ge \frac{2}{r}  \One{\Abs{\cos(2\pi r-\frac{\pi}{4})} \ge 0.9}.
$$
As long as $x>A$ then the quantity in \eqref{constant-lb} is
$$
\ge \int_x^{c_1 x} \frac{2}{r} \One{\Abs{\cos(2\pi r-\frac{\pi}{4})} \ge 0.9} \,dr.
$$
This in turn is bounded from below by
$$
C \int_x^{c_1 x} \frac{1}{r} \,dr \ge C,
$$
as long as the interval of integration contains a full period of the cosine involved.
So we have proved \eqref{constant-lb} for $x$ larger than an absolute constant $C>0$.

For $c_0 \le x \le C$ we only have to observe that the function $\ft{\sigma_1}(r)$ 
does not vanish identically in any interval.
\end{proof}

For a function $f:\RR^2\to\CC$ define the circle discrepancy 
$$
D_t(x) = \int_{C(x,t)} f,
$$
where $C(x,t)$ is a circle of radius $t$ centered at $x$.
We have $D_t(x) = f*\sigma_t (x)$ hence $\ft{D_t}(\xi) = \ft{f}(\xi) \ft{\sigma_t}(\xi)$ so
\beql{ltwo}
\int \Abs{D_t(x)}^2\,dx = \int \Abs{\ft{D_t}(\xi)}^2\,d\xi = \int \Abs{\ft{f}(\xi)}^2 \Abs{\ft{\sigma_t}(\xi)}^2\,d\xi.
\eeq

We now use the following lemma.
\begin{lemma}\label{lm:decay}
For sufficiently large $A>0$ and sufficiently small $a>0$ we have
$$
\int_{\frac{a}{N}<\Abs{\xi}<A} \Abs{\ft{f}(\xi)}^2\,d\xi \ge \frac{1}{3} \int \Abs{\ft{f}(\xi)}^2\,d\xi = \frac13 N^2.
$$
\end{lemma}
\begin{proof}
In \cite{kolountzakis} Lemma \ref{lm:decay} is proved in the form
$$
\int_{\Abs{\xi}<A} \Abs{\ft{f}(\xi)}^2\,d\xi \ge \frac{1}{2} \int \Abs{\ft{f}(\xi)}^2\,d\xi.
$$
To obtain the extra restriction $\Abs{\xi} > a/N$ we notice that for all $\xi$
$$
\Abs{\ft{f}(\xi)} \le \Lone{f} = N^2,
$$
from which it follows that
$$
\int_{\Abs{\xi} \le \frac{a}{N}} \Abs{\ft{f}(\xi)}^2\,d\xi \le C a^2 N^{-2} N^4 = C a^2 N^2.
$$
Now we choose the constant $a>0$ to be sufficiently small.
\end{proof}

Let $\alpha=1/5, \beta=1/4$. Then
\begin{eqnarray*}
\int_{\alpha N}^{\beta N} \int \Abs{D_t(x)}^2 \,dx\, dt &=& \int_{\alpha N}^{\beta N} \int \Abs{\ft{D_t}(\xi)}^2\,d\xi\, dt\\
 &\ge& \int_{\alpha N}^{\beta N}
 \int_{\frac{a}{N} \le \Abs{\xi} \le A} \Abs{\ft{D_t}(\xi)}^2\,d\xi \, dt\\
 &=& \int_{\frac{a}{N} \le \Abs{\xi} \le A} \Abs{\ft{f}(\xi)}^2 \int_{\alpha N}^{\beta N}
	 \Abs{\ft{\sigma_t}(\xi)}^2 \,dt \, d\xi\\
 &=& \int_{\frac{a}{N} \le \Abs{\xi} \le A} \Abs{\ft{f}(\xi)}^2
	\int_{\alpha N}^{\beta N} t^2 \Abs{\ft{\sigma_1}(t\xi)}^2 \,dt \, d\xi\\
 &\ge & \alpha^2 N^2 \int_{\frac{a}{N} \le \Abs{\xi} \le A} \Abs{\ft{f}(\xi)}^2
	\frac{1}{\Abs{\xi}} \int_{\alpha\Abs{\xi}N}^{\beta\Abs{\xi}N} \Abs{\ft{\sigma_1}(u)}^2 \,du \, d\xi.
\end{eqnarray*}
Since $\Abs{\xi} \ge \frac{a}{N}$ in the region of integration it follows that
$\alpha\Abs{\xi}N \ge C$ hence we may use \eqref{constant-lb} to bound
\begin{eqnarray*}
\int_{\alpha N}^{\beta N} \int \Abs{D_t(x)}^2 \,dx dt & \ge &
	C \frac{\alpha^2 N^2}{A} \int_{\frac{a}{N} \le \Abs{\xi} \le A} \Abs{\ft{f}(\xi)}^2\\
 &\ge& C N^4,
\end{eqnarray*}
from Lemma \ref{lm:decay},
which implies that there are $x \in (-2N,2N)^2$ and $t \in (N/5, N/4)$ such that $\Abs{D_t(x)}^2 \ge CN$.
\end{proof}

\section{Proof of the upper bound for the straight line segment}
\label{sec:upper}

In this section we give a proof of Theorem \ref{th:upper}.
First we construct such a coloring function $f$ for fixed $\epsilon$ and then we point out how to construct a single function $f$ which works for all $\epsilon>0$.

We will color the infinite checkerboard so that the inequality 
\beql{length-bound}
\Abs{\int_I f} \le \phi(\Abs{I}) := K \Abs{I}^{1/2+\epsilon}
\eeq
holds for each segment $I$,
where $f$ is the coloring function, which takes the values $+1$ or $-1$
in each cell of the infinite checkerboard.
The constant $K=K_\epsilon$ may depend on $\epsilon$ only.

Let
$$
N_1=2 \mbox{ and }N_{k+1} = N_k  M_k,
$$
where the odd integer $M_k\ge 2$ will be determined later (it will be of the order of some power
of $\log N_k$).
We will inductively color $N_k \times N_k$ checkerboards,
calling the coloring function $f_k$.
These checkerboards will be centered so that their union is the entire plane and
the coloring of the central $N_k \times N_k$ checkerboard will be preserved when we
go to the $N_{k+1} \times N_{k+1}$ checkerboard.
In this fashion a coloring of the entire plane is defined. (See Fig.\ \ref{fig:recursive}.)

\begin{figure}[h]
\begin{center} \resizebox{8cm}{!}{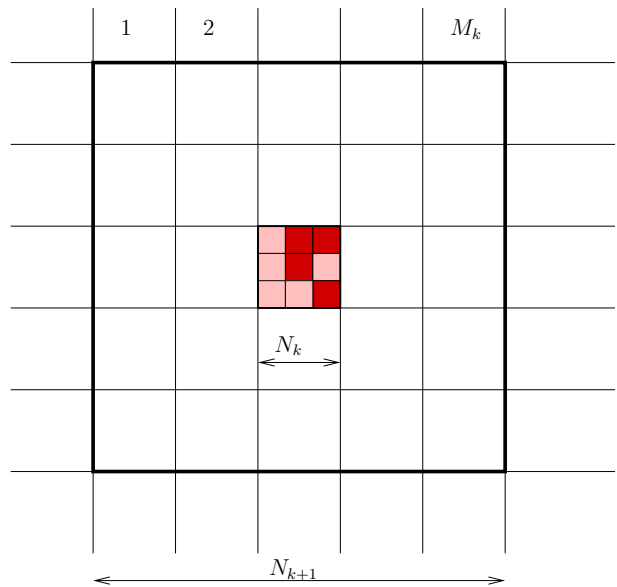} \end{center}
\caption{Extending the coloring from the middle super-cell}
\label{fig:recursive}
\end{figure}

Our $N_{k+1} \times N_{k+1}$ checkerboard consists of $M_k \times M_k$ copies
(super-cells) of a $N_k \times N_k$ checkerboard.
For $\ell > 0$ write $k(\ell)$ for the minimum index $k$ such that $N_k \ge \ell$.
It follows that every segment $I$ in the plane can be broken
into at most, say, 10 segments each of which is entirely contained in a super-cell of size
$N_{k(\Abs{I})} \times N_{k(\Abs{I})}$.

The proposition that we prove by induction on $k$ is:
\begin{quotation}
If the segment $I$ is contained in a single $N_k \times N_k$ super-cell then
\beql{single-cell}
\Abs{\int_I f_k} \le \frac{\phi(N_k)}{100} = \frac{1}{100} K N_k^{\frac12 + \epsilon}.
\eeq
\end{quotation}

Define $f_1$ arbitrarily in the initial $2 \times 2$ checkerboard
and $0$ outside that. Clearly \eqref{single-cell} is valid with, say, $K=500$.
We now assume the existence of a coloring function $f_k$
for which \eqref{single-cell} holds
and will construct the function $f_{k+1}$ to preserve \eqref{single-cell}.

The coloring in the $N_{k+1} \times N_{k+1}$ checkerboard (function $f_{k+1}$) will be defined
using the $N_k \times N_k$ coloring (function $f_k$)
by copying it in each $N_k \times N_k$ super-cell of the
$N_{k+1} \times N_{k+1}$ checkerboard multiplied by a random $\pm 1$ sign.
We insist that the coloring of the central $N_k \times N_k$ super-cell is preserved,
by possibly multiplying everything by $-1$.
We will now show that there exists a choice of the signs
that will ensure that the new function $f_{k+1}$ satisfies \eqref{single-cell}.

As we did in \cite{kolountzakis} we observe that there is a set ${\mathcal S}$ of line segments
in the $N_{k+1} \times N_{k+1}$ checkerboard, of size $O((N_{k+1})^{D})$, $D$ an absolute constant,
such that \eqref{single-cell} holds for {\em all} line segments if it can be proved to hold
for the line segments in ${\mathcal S}$ with the constant $K$ replaced by, say, $K/3$.

Fix now a line segment $I \in {\mathcal S}$.
Let $I_j$ be the intersections of $I$ with the $N_k \times N_k$
super-cells, $j=1,\ldots,n(I)$, $n(I) \le C M_k$,
and let $\epsilon_j = \pm1$ be the corresponding signs used
in the construction of $f_{k+1}$ from copies of $f_k$. We have
$$
\int_I f_{k+1} = \sum_{j=1}^{n(I)} \epsilon_j d_j,\ \mbox{where } d_j = \int_{\widetilde{I_j}} f_k,
$$
and $\widetilde{I_j}$ is the line segment $I_j$ translated
by an element of $N_k\ZZ^2$
to lie in the central $N_k \times N_k$ checkerboard where the function $f_k$ is non-zero.

We now use the following standard estimate (see for instance \cite[Appendix A, Theorem A.1.18]{alon-spencer}):
\begin{othm}\label{th:large-deviations}
{\rm (See for instance \cite[Appendix A, Theorem A.1.18]{alon-spencer}.)}\\
Suppose $X_i$, $i=1,2,\ldots,n$, are independent random variables with $\Mean{X_i} = 0$ and
with the range of each $X_i$ having diameter at most 1.
Write $S=X_1+\cdots+X_n$. Then
$$
\Prob{\Abs{S} > t} \le 2 e^{-2t^2/n}.
$$
\end{othm}
By our inductive hypothesis $\Abs{d_j} \le \phi(N_k)/ 100$.
Therefore, we obtain from Theorem \ref{th:large-deviations}
the inequality 
$$
\Prob{\Abs{\sum_{j=1}^{n(I)} \epsilon_j d_j} > t\cdot 2\phi(N_k)/ 100} \le 2e^{-2t^2/n(I)},
$$
and choosing $t = C(M_k \log N_k)^{1/2}$ for a sufficiently large constant $C$
(depending on the value of $D$ only) we obtain
\beql{estimate}
\Prob{\Abs{\int_I f_{k+1}} \ge C \sqrt{M_k \log N_k} \phi(N_k)} \le C (N_k)^{-(D+1)}.
\eeq
This guarantees that, with positive probability, all segments in ${\mathcal S}$
have discrepancy bounded above by
$$
C \sqrt{M_k\log N_k} \phi(N_k).
$$

We still have to ensure that
$$
C \sqrt{M_k\log N_k} \phi(N_k) \le \frac{1}{3} \phi(N_{k+1}) / 100.
$$
This is easily seen to follow from the choice $M_k = C \log^{1/(2\epsilon)} N_k$ for large enough $C>0$.
This completes the inductive proof of \eqref{single-cell}.

Let now $I$ be any line segment in the plane. Then $I$ can be broken up into at most $3$ segments $J$ with $J$ being contained in a single $N_k\times N_k$ super-cell, where $k = k(\Abs{I})$.
Since $N_k \le C_\delta \Abs{I}^{1+\delta}$ for any $\delta>0$ it follows that the discrepancy of $I$ is at most $C_\epsilon' \Abs{I}^{\frac{1}{2}+\epsilon'}$ for any $\epsilon' > \epsilon$. This concludes the proof of how to construct a coloring $f$ which satisfies \eqref{length-bound} for a single $\epsilon>0$. If one desires to have a coloring $f$ which satisfies \eqref{length-bound} for all $\epsilon>0$ (with different constants of course) all one has to do is to let $\epsilon \to 0$ slowly during the construction. 

\section{Remarks about $L^p$ bounds on the discrepancy function}\label{sec:lp}

Let us restrict our attention to a finite $N\times N$ checkerboard coloring with coloring function
$f:\RR^2\to\Set{-1,+1}$.
So far, in this paper and in \cite{kolountzakis}, we have examined two discrepancy functions: the
circle discrepancy $D_t(x)$ defined in \S\ref{sec:lower} as well as the line discrepancy function studied
in \cite{kolountzakis} and in \S\ref{sec:upper} of this paper. Let us denote the line discrepancy function
by $\Delta_u(x)$ where $u \in S^1$ is a unit vector and $x \in \RR$:
$$
\Delta_u(x) = (\pi_u f)(x) = \int f(x\cdot u + y\cdot u^\perp)\,dy.
$$
Here $\pi_u f$ denotes the one-variable projection of $f$ onto the line defined by $u$ and $u^\perp$
is a unit vector orthogonal to $u$.

In \cite{kolountzakis} and in \S\ref{sec:lower} of this paper we proved that
$$
\sup_{u,x} \Abs{\Delta_u(x)} \ge C N^{1/2}\mbox{\ and\ }
 \sup_{t,x} \Abs{D_t(x)} \ge C N^{1/2}.
$$
It is natural and customary in the field of discrepancy bounds \cite{beck-chen,matousek}
to study several measures of size for the discrepancy functions, such as their $L^p$ norms.
An appropriate way to define the $L^p$ discrepancy of $f$ for the two cases we are studying
is the following.
\begin{definition}
The circle $L^p$ discrepancy ($1\le p < \infty$) of a coloring $f$ is
\beql{lp-circle}
D(f,p) = \left( \frac{1}{N^3} \int_{N/5}^{N/4} \int \Abs{D_t(x)}^p \,dx\,dt \right)^{1/p}.
\eeq
The line $L^p$ discrepancy is
\beql{lp-line}
\Delta(f,p) = \left( \frac{1}{N} \int_{S^1} \int \Abs{\Delta_u(x)}^p \,dx\,du \right)^{1/p}.
\eeq
The sup discrepancies are $D(f,\infty) = \sup_{x,t} \Abs{D_t(x)}$ and
$\Delta(f,\infty) = \sup_{u,x} \Abs{\Delta_u(x)}$.
\end{definition}
The reason for the factors $N^{-3}$ and $N^{-1}$ in \eqref{lp-circle} and \eqref{lp-line}
is to almost normalize the measure and make the different $L^p$ norms comparable. Indeed, in the case of
$D_t(x)$ (circle) the range for $t$ is $\sim N$ and that of $x$ is $\sim N^2$, and in the case of
$\Delta_u(x)$ the range of $u$ is constant while that of $x$ is $\sim N$ (beyond these bounds the functions obviously vanish).
With these definitions we have $D(f,p_1) \le C_{p_1, p_2} D(f,p_2)$ for any $p_1 < p_2$ ($\infty$ included) and similarly for the line discrepancies.

We now point out that the proof given in \S\ref{sec:lower} is essentialy that
$$
D(f,2) \ge C N^{1/2},
$$
which clearly implies the same bound for the sup discrepancy $D(f,\infty)$.
Similarly in \cite{kolountzakis} we not only proved that $\Delta(f,\infty) \ge C N^{1/2}$
but also that $\Delta(f,2) \ge C N^{1/2}$.

To summarize, the behavior of $D(f,p)$ and $\Delta(f,p)$ is now essentially known (at least up to
logarithmic factors) for $p\ge 2$. Unfortunately not much is known for $p<2$. The following example shows
that the study of these quantities is probably a lot harder that the $L^2$ discrepancies.
Color the $N \times N$ checkerboard as shown in Fig.\ \ref{fig:lone}, that is paint each row with
a single color and alternate these colors from each row to the next.
This is a horrible example as far as $\Delta(f,\infty)$ is concerned. It is easily seen that
$\Delta(f,\infty) \ge C\cdot N$ since every horizontal line contains one color only.
However $\Delta(f,1)$ is a lot smaller.
\begin{figure}[h]
 \begin{center} \resizebox{5cm}{!}{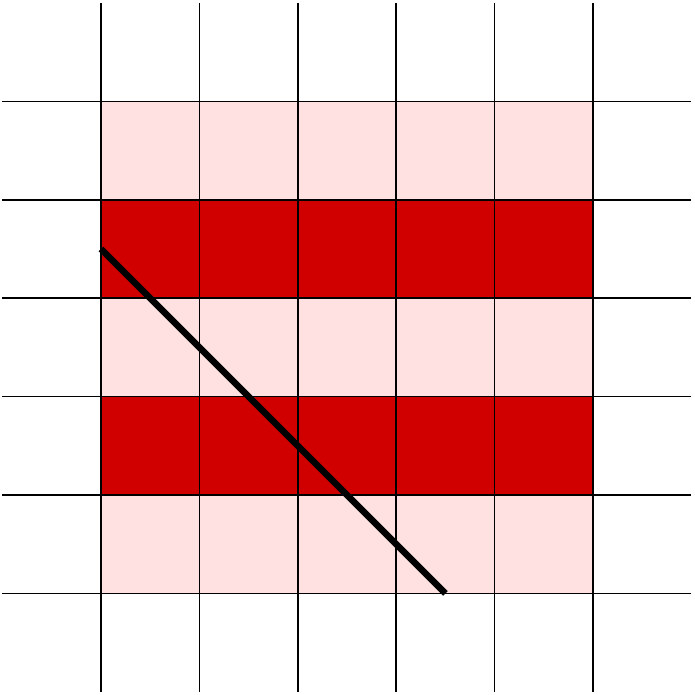} \end{center}
\caption{A coloring with $\Delta(f,1) \sim \log N$}
\label{fig:lone}
\end{figure}
\begin{theorem}\label{th:lone}
For the coloring shown in Fig.\ \ref{fig:lone} we have $\Delta(f,1) \sim \log N$.
\end{theorem}
\begin{proof}
This is a straightforward calculation. We first observe that for each line the contribution to
the discrepancy will come only at the ends of the segment, where the line intersects the boundary.
Indeed if a line traverses a number $k$ of whole rows then they either contribute 0 to the discrepancy of
the line (if $k$ is even) or at most one of them contributes (if $k$ is odd) an amount equal
to the length of the intersection of the line with one such row (the quantity $L(\theta)$ below).

Therefore the discrepancy of a line will be at most a constant times the length it cuts from a horizontal strip of width 1 and is of course always at most $CN$. If a line makes angle $\theta$ with the $x$-axis then it intersects such a strip at length
$$
L(\theta) = \frac{1}{\sin\theta}.
$$
When computing the $L^1$-discrepancy and we carry out the direction integration first we are computing (four times) the integral
$$
\int_{0}^{\pi/2} \min\Set{\frac{1}{\sin\theta}, CN} \,d\theta \sim C\log N,
$$
since $CN$ becomes the minimum from $\theta=0$ to $\theta\sim 1/N$.
Integration then over $x$ will contribute $CN$ which is cancelled by the normalization in the definition of $\Delta(f,1)$. This proves the lower bound $\Delta(f,1) \ge C \log N$.

To prove the upper bound one keeps $\theta$ fixed and integrates along $x$ first. The contribution to the integral will be $\ge C N / \sin\theta$. One then integrates for $\theta \in (1/N, \pi/2)$ and normalizes
to obtain the lower bound $\ge C \log N$.
\end{proof}

\section{Open problems}\label{sec:open-problems}

We finish with a list of questions. We owe questions (\ref{it:fixed-radius}) and (\ref{it:periodic-coloring}) to \href{http://www.renyi.hu/~emarci/}{M\'arton Elekes}.
\begin{enumerate}
\item 
What is the true order of magnitude of $\Delta(f,1)$ and of $\Delta(f,p)$ for $1 \le p \le 2$,
at least up to logarithmic factors?

We believe that $\Delta(f,1) \ge C \log N$ but it seems that this would be substantially harder to prove than the current $L^2$ lower bounds.

\item
Is there a coloring $f$ for which $D(f,1)$ is much less than $N^{1/2}$? In other words, is there a coloring which plays for circular arcs the role that the coloring of Fig.\ \ref{fig:lone} plays for the lines?

We believe not.

\item
Obtain a lower bound for the discrepancy of a shape that consists of straight line segments, for example an orthogonal isosceles L-shape which is free to translate, dilate and rotate in the plane.

In other words, although each of the two arms of the L-shape must have a large discrepancy at some placements, it could be that, for some colorings, the two arms conspire to cancel each other's discrepancy.

\item\label{it:fixed-radius}
In Theorem \ref{th:arcs} can one restrict the radius of the circle to equal to, say, $N/5$, instead of just lying in the interval $(N/5, N/4)$?

We believe yes.

\item\label{it:periodic-coloring}
Can one find a periodic coloring for which the upper bound of Theorem \ref{th:upper} holds?

Probably not.

\end{enumerate}


\end{document}